\documentclass[10pt, reqno]{amsart}
\usepackage{amsfonts}
\usepackage{ifthen}
\usepackage{amsthm}
\usepackage{amsmath}
\usepackage{graphicx}
\usepackage{caption}
\usepackage{subcaption}
\usepackage{amscd,amssymb,amsthm}
\usepackage{color}
\usepackage{hyperref}
\usepackage{array,multirow,makecell}
\usepackage{cite,enumerate}
\setcellgapes{1pt}
\makegapedcells
\newcolumntype{R}[1]{>{\raggedleft\arraybackslash }b{#1}}
\newcolumntype{L}[1]{>{\raggedright\arraybackslash }b{#1}}
\newcolumntype{C}[1]{>{\centering\arraybackslash }b{#1}}
\newcounter{minutes}
\divide\time by 60
\newcounter{hours}
\multiply\time by 60 \addtocounter{minutes}{-\time}

\newtheorem{theorem}{Theorem}[section]
\newtheorem{lemma}[theorem]{Lemma}
\newtheorem{definition}[theorem]{Definition}
\newtheorem{corollary}[theorem]{Corollary}

\newtheorem{example}[theorem]{Example}

\title[Geometric Properties of the Four Parameters Wright Function]{Geometric Properties of the Four Parameters Wright Function}

\author[S. Das \;and\; K. Mehrez]{Sourav Das\;and\;Khaled Mehrez}
\address{Sourav Das\newline Department of Mathematics, National Institute of Technology Jamshedpur,\newline Jamshedpur-831014, Jharkhand, India.}
\email{souravdasmath@gmail.com, souravdas.math@nitjsr.ac.in}
\address{Khaled Mehrez\newline D\'epartement de Math\'ematiques, Facult\'e de Sciences de Tunis,\newline Universit\'e Tunis El Manar, Tunisia.}
\email{k.mehrez@yahoo.fr}

\keywords{Wright function, analytic function, univalent function, starlike function, close-to-convex function}

\subjclass[2010]{30C45, 30D15, 33C10}

\begin{document}


\maketitle

\begin{abstract}
In this paper, four parameters Wright function is considered. Certain geometric properties such as starlikeness, convexity, uniform convexity and close-to-convexity are discussed for this function.
Further, certain geometric properties of normalized Bessel function of the first kind and two parameters Wright function are studied as a consequence. Interesting corollaries and examples are provided to support that these results are better than the existing ones and improve several results available in the literature.
\end{abstract}

\section{Introduction}
The Wright function
\begin{equation}\label{1.2}
W_{\alpha,\beta}(z)=\sum_{k=0}^\infty \frac{z^k}{k!\Gamma(\alpha k+\beta)},\;\beta,z\in\mathbb{C}, \alpha>-1
\end{equation}
was introduced by E. M. Wright \cite{57} in connection with the asymptotic theory of partitions.
The Wright function  plays vital role in
fractional calculus \cite{Podlubny-Fractional-Book-99,Luchko-Gorenflo-1998}, the Mikusi\'ski operational calculus, integral transforms of the Hankel type and stochastic processes. For a historical overview regarding the Wright function and its applications we refer to
\cite[ Appendix F]{35}.
Note that $W_{\alpha,\beta}(z)$ is an entire function of order $1/(1+\alpha)$ and also known as the generalized Bessel function \cite{Podlubny-Fractional-Book-99,Baricz-Book-Bessel-Function1994}. These functions
generalize hypergeometric functions \cite{abramowitz64,Andrews-Askey-Roy-Book-specialFunctions-1999}.

The four parameters Wright function \cite{Luchko-Gorenflo-1998,Mehrez-Fox-Wright-JMAA-2019}
\begin{align}\label{eq:WrightFn-four-parameter}
\mathcal{W}_{(\mu,a),(\nu,b)}(z)=\sum_{k=0}^{\infty}\dfrac{z^k}{\Gamma(a+k\mu)\Gamma(b+k\nu)},\quad a,b\in \mathbb{C}, \quad  \mu,\nu\in\mathbb{R},
\end{align}
was studied by E. M. Wright for the case $\mu,\nu>0$ in \cite{Wright-JLMS1935}. Further, he derived several properties of  $\mathcal{W}_{(\mu,a),(\nu,b)}(z)$ for the case $b=\nu=1$ and $-1<\mu<0$ in \cite{Wright-QJM1940}.
It can be verified \cite{Mehrez-Fox-Wright-JMAA-2019} that if $\mu+\nu>0$, then the infinite series expansion \eqref{eq:WrightFn-four-parameter} of $\mathcal{W}_{(\mu,a),(\nu,b)}(z)$ converges absolutely for all $z\in\mathbb{C}$.
It is well-known (see \cite{Luchko-Gorenflo-1998}) that $\mathcal{W}_{(\mu,a),(\nu,b)}(z)$
is an entire function for $a,b\in\mathbb{C}$ and $0<-\mu<\nu.$

Let us consider that $\mathcal{H}$ denotes the class of all analytic functions inside the unit disk $\mathbb{D}=\{z:|z|<1\}$.
Suppose that $\mathcal{A}$ is the class of all functions $f\in \mathcal{H}$ which are normalized by $f(0)=f'(0)-1=0$ such that
$$f(z)=z+\sum_{k=2}^{\infty}a_kz^k,\quad z\in\mathbb{D}.$$

A function $f\in\mathcal{A}$ is said to be a starlike function (with respect to the origin $0$) in $\mathbb{D}$, if  $f$ is univalent in $\mathbb{D}$ and
$f(\mathbb{D})$ is a star-like domain with respect to $0$ in $\mathbb{C}$. This class of starlike functions is denoted by $\mathcal{S}^*.$ The analytic characterization of $\mathcal{S}^*$ is given \cite{Duren-book-Univalent-Functions-1983} below:
$$
\Re\left(\frac{zf^\prime(z)}{f(z)}\right)>0\;\;\forall z\in\mathbb{D}\;\;\;\Longleftrightarrow\;\;\;\;f\in\mathcal{S}^*.$$
Let $\eta\in[0,1)$ and $z\in \mathbb{D}$. If
$
\Re\left(\dfrac{zf'(z)}{f(z)} \right)>\eta,
$
then the function $f\in\mathcal{A}$ is said to be a starlike function of order $\eta$.
We denote the class of starlike functions of order $\eta$ by $\mathcal{S}^*(\eta)$.

A function $f\in\mathcal{A}$ is said to be  convex  in $\mathbb{D}$ if $f$ is univalent in $\mathbb{D}$ and $f(\mathbb{D})$ is a convex domain in $\mathbb{C}$. We denote this class of convex functions by $\mathcal{K}$. This class can be analytically characterized as follows:
$$\Re\left(1+\frac{zf^{\prime\prime}(z)}{f^\prime(z)}\right)>0,\;\forall z\in\mathbb{D} \;\;\;\Longleftrightarrow\;\;\;\;f\in \mathcal{K}.$$
A function $f(z)\in\mathcal{A}$ is said to be a convex function of order $\eta$ $(0\leq \eta<1),$  if
$$
\Re\left(1+\dfrac{zf''(z)}{f'(z)} \right)>\eta,\quad z\in \mathbb{D}.
$$
This class is denoted by $\mathcal{K}(\eta)$. In particular, $\mathcal{K}=\mathcal{K}(0)$ and $\mathcal{S}^*=\mathcal{S}^*(0).$
It is well-known that $zf'$ is starlike  if and only if   $f\in\mathcal{A}$ is convex.
A function $f(z)\in\mathcal{A}$ is said to be close-to-convex in $\mathbb{D}$ if $\exists$ a starlike function $g(z)$ in $\mathbb{D}$ such that
$$
\Re\left(\dfrac{zf'(z)}{g(z)}\right)>0,\quad z\in \mathbb{D}.
$$
The class of all close-to-convex functions is denoted by $\mathcal{C}$. It can be easily verified that
$\mathcal{K}\subset \mathcal{S}^{*}\subset\mathcal{C}$. It is well-known that every close-to-convex function in $\mathbb{D}$ is
also univalent in $\mathbb{D}$.

A function $f\in\mathcal{A}$ is said to be uniformly convex (starlike) if for every circular arc $\gamma$ contained in $\mathbb{D}$ with center
$\zeta\in\mathbb{D}$ the image arc $f(\gamma)$  is convex (starlike w.r.t. the image $f(\zeta)$).
The class of all uniformly convex (starlike) functions is denoted by $UCV$ ($UST$) \cite{Ronning-PAMS-Uniformly-Convex-93}.
In \cite{Goodman-UCV-APM1991,Goodman-UST-JMAA1991}, A. W. Goodman introduced these classes. Later, F. R\o nning
\cite{Ronning-PAMS-Uniformly-Convex-93} introduced a new
class of starlike functions $\mathcal{S}_p$  defined by
$$
\mathcal{S}_p:=\{f:f(z)=zF'(z),\,F\in UCV   \}.
$$
For further details on geometric properties of analytic functions we refer to
\cite{Duren-book-Univalent-Functions-1983,Goodman-book-I-1983,Goodman-book-II-1983,Baricz-Book-Bessel-Function1994,MacGregor-PAMS1963,MacGregor-PAMS1964,Raina-RSMMUP1996, raza1,Mocanu-starlike-cond-LibMath1993}
and references cited therein.

Problems for investing geometric properties including starlikeness, closed-to-convexity, convexity or univalency of family of analytic functions in
the $\mathbb{D}$ involving special functions have always been attracted by several researchers
\cite{Prajapat-Wright-Function-ITSF2015,Bansal-Prajapat-Mittagleffler-2016,Duren-book-Univalent-Functions-1983,Goodman-book-I-1983,Goodman-book-II-1983,
Baricz-Book-Bessel-Function1994,MacGregor-PAMS1963,MacGregor-PAMS1964,Ozaki-1935,Wilf-Subord-convex-map-PAMS1961,AS-CMA-2010}.

We observe that  $\mathcal{W}_{(\mu,a),(\nu,b)}(z) \notin \mathcal{A}$. For this reason,
we consider the following normalization of $\mathcal{W}_{(\mu,a),(\nu,b)}(z)$ as follows
\begin{equation}\label{eq:normalized-Wright-four-parameter}
\begin{split}
\mathbb{W}_{(\mu,a),(\nu,b)}(z)&=z\Gamma(a)\Gamma(b)\mathcal{W}_{(\mu,a),(\nu,b)}(z)\\
&=\sum_{k=0}^{\infty}\dfrac{\Gamma(a)\Gamma(b)z^{k+1}}{\Gamma(a+k\mu)\Gamma(b+k\nu)},\;  a,b\in \mathbb{C},\, \mu,\nu\in\mathbb{R}\\
 &=\sum_{k=0}^{\infty}\alpha_k z^{k+1},
\end{split}
\end{equation}
where
$$\alpha_k=\frac{\Gamma(a)\Gamma(b)}{\Gamma(a+k\mu)\Gamma(b+k\nu)}.$$
Although in \eqref{eq:normalized-Wright-four-parameter},  $a,b,z\in\mathbb{C}$, however in this work $a$ and $b$ are restricted to real valued
and $z\in\mathbb{D}$.

For two functions $f$ and $g$, which are analytic in $\mathbb{D}$, we say that the function the function $f(z)$ is subordinate to $g(z)$ in $\mathbb{D}$, and write $f(z)\prec g(z)$ or $f\prec g$ $(z\in\mathbb{D})$, if there exists a function $w(z)$, which is analytic in $\mathbb{D}$ with $w(0)=0$ and $|w(z)|<1$ for all $z\in \mathbb{D}$, such that $f(z)=g(w(z))$, $z\in\mathbb{D}$. It is well-known that
if $f(z)\prec g(z)$ $(z\in\mathbb{D})$, then $f(0)=g(0)$ and $f(\mathbb{D}) \subset g(\mathbb{D})$.
Furthermore, if the function $g(z)$ is univalent in $\mathbb{D}$, then $f(z)\prec g(z)$ if and only if $f(0)=g(0)$ and
$f(\mathbb{D})\subset g(\mathbb{D})$.
Following definition will be very helpful to prove some of the main results.
\begin{definition}
Let $\{\beta_n\}_{n\geq1}$ be a sequence of complex numbers. Then $\{\beta_n\}_{n\geq1}$ is called a subordinating factor sequence, if
\begin{align}\label{eq:series-subord}
f(z)=\sum_{n=1}^{\infty}\alpha_n z^n \in \mathcal{K}
\end{align}
implies
\begin{align}\label{eq:series-subord-cond}
\sum_{n=1}^{\infty}\alpha_n\beta_n z^n \prec f(z).
\end{align}
This class is denoted by $\mathcal{F}.$
A finite sequence $\{\beta_n\}_{n=1}^{k}$  is called a subordinating factor sequence if \eqref{eq:series-subord} yields  \eqref{eq:series-subord-cond}
whenever $\alpha_{k+1}=\alpha_{k+2}=\cdots =0$. This class of such finite sequences of length $k$ is denoted by $\mathcal{F}_{k}$.
\end{definition}

This paper is organized as follows. We provide some lemmas in Section 2, which will be helpful to prove the main results.
In Section 3, starlikeness, convexity and uniform convexity of $\mathcal{W}_{(\mu,a),(\nu,b)}(z)$, are discussed using the properties of Fox-Wright function. In Section 4, we provide some alternative conditions for starlikeness, convexity and uniform convexity of $\mathbb{W}_{(\mu,a),(\nu,b)}(z),$ which will be helpful to discuss close-to-convexity (univalency) of $\mathbb{W}_{(\mu,a),(\nu,b)}(z).$
In Section 5, we derive some properties and inequalities related to $\mathcal{W}_{(\mu,a),(\nu,b)}(z)$ and $\mathbb{W}_{(\mu,a),(\nu,b)}(z)$ involving hypergeometric function. 
 In Section 6, we discuss geometric properties of normalized Bessel function of the first kind and two parameters Wright function,
as applications and show that the results obtained in this paper, are better than the existing ones available in the literature.

\section{Some Lemmas}
In this section, we provide some useful lemmas which will be useful to complete the
proof of the main results.

\begin{lemma}[\cite{Fejer1936}]\label{lm:convex-decreasing-R}
Let $\{a_k\}_{k=1}^{\infty}$ be a sequence of non-negative real numbers such that $a_1=1$. If $\{a_k\}_{k=2}^{\infty}$ is convex decreasing, i.e.,
$0\geq a_{k+2}-a_{k+1}\geq a_{k+1}-a_{k}$, then
$$
\Re\left(\sum_{k=1}^{\infty}a_kz^{k-1}\right)>\dfrac{1}{2}, \quad z\in \mathbb{D}.
$$
\end{lemma}
%

\begin{lemma}[\cite{Wilf-Subord-convex-map-PAMS1961}]\label{lm:equivalent-cond}
Let $\{\gamma_k\}$ be a sequence of complex numbers and $z\in\mathbb{D}$. Then the following statements are equivalent:
\begin{enumerate}[(i)]
\item $\{\gamma_k\}_{k=1}^{\infty}\in \mathcal{F}.$
\item $\Re\left(1+2\sum_{k=1}^{\infty}\gamma_kz^k \right)>0$.
\end{enumerate}
\end{lemma}

\begin{lemma}[\cite{Mocanu-starlike-cond-RRMPA-1988}]\label{lm:starlike-derivative-mocannu}
Let $f(z)\in\mathcal{A}$ and $|f'(z)-1|<2/\sqrt{5}\;\; \forall\,z\in \mathbb{D}.$ Then $f(z)$ is a starlike function in $\mathbb{D}$.
\end{lemma}

\begin{lemma}[\cite{MacGregor-PAMS1963}]\label{lm:univalent-D1/2}
Suppose that $f(z)\in\mathcal{A}$ and $|(f(z)/z)-1|<1 \;\; \forall\,z\in \mathbb{D}$. Then $f(z)$ is a univalent and starlike in $\mathbb{D}_{1/2}=\{z:|z|<1/2\}$.
\end{lemma}

\begin{lemma}[\cite{MacGregor-PAMS1964}]\label{lm:convex-D1/2}
Let $f(z)\in\mathcal{A}$ and $|f'(z)-1|<1\;\;\forall\, z\in \mathbb{D}.$ Then $f(z)$ is a convex function in $\mathbb{D}_{1/2}=\{z:|z|<1/2\}$.
\end{lemma}



\begin{lemma}[\cite{Ravichandran-Ganita-UCV-2002}]\label{lm:ucv-Sp}
Let $f(z)\in\mathcal{A}$.
\begin{enumerate}[(i)]
  \item If $\displaystyle\left|\frac{zf''(z)}{f'(z)}\right|<\frac{1}{2}$, then $f(z)\in UCV$.
  \item If $\displaystyle\left|\frac{zf'(z)}{f(z)}-1\right|<\frac{1}{2}$, then $f(z)\in \mathcal{S}_p$.
\end{enumerate}
\end{lemma}

\begin{lemma}\label{lm:two-ineq}
For any $a,b>0$, the following inequalities hold:
\begin{align}\label{Eq35}
\dfrac{k}{a(a+1)\cdots(a+k-1)}&\leq\dfrac{1}{a(a+1)^{k-2}},\quad k\in\mathbb{N}\setminus\{1\},\\\label{Eq36}
\dfrac{1}{b(b+1)\cdots(b+k-1)}&\leq\dfrac{1}{b(b+1)^{k-1}},\quad k\in\mathbb{N}.
\end{align}
\end{lemma}
\begin{proof}
Under the given hypothesis, it can be observed that
\begin{align}\label{eq:ineq-part-1}
1\left(1+\dfrac{1}{a+1} \right)\left(1+\dfrac{2}{a+1} \right)\cdots \left(1+\dfrac{k-3}{a+1} \right) \left(1+\dfrac{a-1}{k} \right)\geq1.
\end{align}
Multiplying both sides of \eqref{eq:ineq-part-1} by $a(a+1)^{k-2}$, we obtain
\begin{align*}
a(a+1)(a+2)\cdots (a+k-2)\dfrac{(a+k-1)}{k}\geq a(a+1)^{k-2},\quad\mbox{ for }k\geq2,
\end{align*}
which proves the inequality  \eqref{Eq35}.

It can be noted that under the given hypothesis, following inequality holds true:
\begin{align}\label{eq:ineq-part-2}
1\left(1+\dfrac{1}{b+1} \right)\left(1+\dfrac{2}{b+1} \right)\cdots \left(1+\dfrac{k-2}{b+1} \right)\geq1.
\end{align}
Multiplying both sides of \eqref{eq:ineq-part-2} by $b(b+1)^{k-1}$, we obtain
$$
b(b+1)(b+2)\cdots (b+k-1)\geq b(b+1)^{k-1},\quad\mbox{ for }k\geq1,
$$
which proves the inequality \eqref{Eq36}.
\end{proof}
\section{Starlikeness, convexity and uniform convexity}\label{sec-new-method}
To prove some of the main results, we need the Fox-Wright function ${}_p\Psi_q[z]$, defined by \cite[p. 4, Eq. (2.4)]{Wright-Gen-Bessel-PLMS-1935}
\begin{equation}\label{11}
{}_p\Psi_q\Big[_{(b_1,B_1),...,(b_q,  B_q)}^{(a_1,A_1),...,(a_p,  A_p)}\Big|z \Big]={}_p\Psi_q\Big[_{(b_q,  B_q)}^{(a_p,  A_p)}\Big|z \Big]=\sum_{k=0}^\infty\frac{\prod_{i=1}^p\Gamma(a_i+kA_i)}{\prod_{j=1}^q\Gamma(b_j+kB_j)}\frac{z^k}{k!},
\end{equation}
where $A_i$, $B_j\in \mathbb{R}^+$  $(i=1,...,p,j=1,...,q)$ and $a_i, b_j\in \mathbb{C}.$
The series (\ref{11}) converges uniformly and absolutely for all bounded $|z|, z\in \mathbb{C}$ when
$$\epsilon=1+\sum_{j=1}^q B_j-\sum_{j=1}^p A_j>0.$$

In \cite[Theorem 4]{PS}, the authors established the following two-sided inequality
\begin{equation}\label{fi}
\psi_0e^{\psi_1\psi_0^{-1}|z|}\leq{}_p\Psi_q\Big[_{(b_q,  B_q)}^{(a_p,  A_p)}\Big|z \Big]\leq\psi_0-(1-e^{|z|})\psi_1,
\end{equation}
is valid for all $z\in \mathbb{R}$ and for all ${}_p\Psi_q[z]$ satisfying
\begin{equation}
\psi_1>\psi_2\;\;\textrm{and}\;\;\psi_1^2<\psi_0\psi_2.
\end{equation}
Here,
$$\psi_k=\frac{\prod_{j=1}^p\Gamma(a_j+kAj)}{\prod_{j=1}^q\Gamma(b_j+kB_j)},\;k=0,1,2.$$

\begin{theorem}\label{KT1}
Assume that $a, b, \mu,\nu>0$. Suppose that the following conditions hold:
\begin{displaymath}
(H_1):\left\{ \begin{array}{ll}
(i)& \frac{\Gamma(a+2\mu)}{\Gamma(a+3\mu)}<  \frac{\Gamma(b+3\nu)}{3\Gamma(b+2\nu)},\\
(ii) & \frac{\Gamma(a+\mu)\Gamma(a+3\mu)}{\Gamma^2(a+2\mu)}<\frac{3\Gamma^2(b+2\nu)}{2\Gamma(b+\nu)\Gamma(b+3\nu)},\\
(iii)& \frac{2\Gamma(a)\Gamma(b)}{\Gamma(a+\mu)\Gamma(b+\nu)}+\frac{3(e-1)\Gamma(a)\Gamma(b)}{\Gamma(a+2\mu)\Gamma(b+2\nu)}<1.
\end{array} \right.
\end{displaymath}
Then the function $\mathbb{W}_{(\mu,a),(\nu,b)}(z)$ is starlike in $\mathbb{D}.$
\end{theorem}
\begin{proof}
Let us assume that $q(z)=\dfrac{z\mathbb{W}'_{(\mu,a),(\nu,b)}(z)}{\mathbb{W}_{(\mu,a),(\nu,b)}(z)}$, $z\in\mathbb{D}$. Then $q(z)$ is analytic in $\mathbb{D}$ and
$q(0)=1$. To prove that $\Re(q(z))>0$, it is enough to show that $|q(z)-1|<1$.

From \eqref{eq:normalized-Wright-four-parameter}, we get
\begin{equation}\label{K1}
\begin{split}
\left|\mathbb{W}^\prime_{(\mu,a),(\nu,b)}(z)-\frac{\mathbb{W}_{(\mu,a),(\nu,b)}(z)}{z}\right|&\leq \sum_{k=0}^\infty \frac{(k+1)\Gamma(a)\Gamma(b)}{\Gamma(a+\mu+k\mu)\Gamma(b+\nu+k\nu)}\\
&=\Gamma(a)\Gamma(b){}_1\Psi_2\Big[ \begin{array}{c}(2, 1)\\(a+\mu, \mu), (b+\nu, \nu)
			   \end{array} \Big|1\Big].
\end{split}
\end{equation}
In our case
$$\psi_0=\frac{\Gamma(2)}{\Gamma(a+\mu)\Gamma(b+\nu)},\;\psi_1=\frac{\Gamma(3)}{\Gamma(a+2\mu)\Gamma(b+2\nu)},\;\textrm{and}\;\psi_2=\frac{\Gamma(4)}{\Gamma(a+3\mu)\Gamma(b+3\nu)}.$$
It is easy to see that the hypotheses $``(H_1): (i), (ii)"$ are equivalent to $\psi_2<\psi_1$ and $\psi_1^2<\psi_0\psi_2$, and consequently
\begin{equation}\label{eq:kkk2}
{}_1\Psi_2\Big[ \begin{array}{c}(2, 1)\\(a+\mu, \mu), (b+\nu, \nu)
			   \end{array} \Big|1\Big]\leq\frac{1}{\Gamma(a+\mu)\Gamma(b+\nu)}-\frac{2(1-e)}{\Gamma(a+2\mu)\Gamma(b+2\nu)}.
\end{equation}
Combining \eqref{K1} and \eqref{eq:kkk2}, we obtain
$$
\left|\mathbb{W}^\prime_{(\mu,a),(\nu,b)}(z)-\frac{\mathbb{W}_{(\mu,a),(\nu,b)}(z)}{z}\right|< \frac{\Gamma(a)\Gamma(b)}{\Gamma(a+\mu)\Gamma(b+\nu)}-\frac{2(1-e)\Gamma(a)\Gamma(b)}{\Gamma(a+2\mu)\Gamma(b+2\nu)},\;\textrm{for\;all}\;z\in\mathbb{D}.
$$

Now, using \eqref{eq:normalized-Wright-four-parameter} and the triangle inequality $|z_1+z_2|\geq||z_1|-|z_2||,$ we have
\begin{equation}\label{ratio-w-str}
\begin{split}
\left| \dfrac{\mathbb{W}_{(\mu,a),(\nu,b)}(z)}{z}\right|
&\geq 1-\left| \sum_{k=1}^{\infty}\dfrac{\Gamma(a)\Gamma(b)z^k}{\Gamma(a+k\mu)\Gamma(b+k\nu)} \right|\\
&\geq1- \sum_{k=0}^{\infty}\dfrac{\Gamma(a)\Gamma(b)}{\Gamma(a+\mu+k\mu)\Gamma(b+\nu+k\nu)} \\
&=1- \Gamma(a)\Gamma(b)\sum_{k=0}^{\infty}\dfrac{\Gamma(k+1)}{\Gamma(a+\mu+k\mu)\Gamma(b+\nu+k\nu)k!}\\
&=1-\Gamma(a)\Gamma(b) {}_1\Psi_2\Big[ \begin{array}{c}(1, 1)\\(a+\mu, \mu), (b+\nu, \nu)
			   \end{array} \Big|1\Big].
\end{split}
\end{equation}
In this case, $\psi_0=\frac{1}{\Gamma(a+\mu)\Gamma(b+\nu)}$, $\psi_1=\frac{1}{\Gamma(a+2\mu)\Gamma(b+2\nu)}$
and $\psi_2=\frac{2}{\Gamma(a+3\mu)\Gamma(b+3\nu)}$. Clearly, the inequalities $\psi_2<\psi_1$ and $\psi_1^2<\psi_0\psi_2$
are satisfied under the given hypothesis. Using \eqref{fi}, we have
\begin{align}\label{ineq-1-psi-2}
{}_1\Psi_2\Big[ \begin{array}{c}(1, 1)\\(a+\mu, \mu), (b+\nu, \nu)
			   \end{array} \Big|1\Big]
\leq \frac{1}{\Gamma(a+\mu)\Gamma(b+\nu)}+\frac{(e-1)}{\Gamma(a+2\mu)\Gamma(b+2\nu)}
\end{align}
Combining \eqref{ratio-w-str} and \eqref{ineq-1-psi-2}, we have
\begin{align}\label{w-z}
\left| \dfrac{\mathbb{W}_{(\mu,a),(\nu,b)}(z)}{z}\right|\geq
1- \left[\frac{\Gamma(a)\Gamma(b)}{\Gamma(a+\mu)\Gamma(b+\nu)}  + \frac{(e-1)\Gamma(a)\Gamma(b)}{\Gamma(a+2\mu)\Gamma(b+2\nu)}\right]>0,
\end{align}
under the given hypothesis $(H_1:(iii))$.
Now, using \eqref{K1} and \eqref{w-z}, we obtain
\begin{align*}
\left|q(z)-1 \right|&=\left|\dfrac{z\mathbb{W}'_{(\mu,a),(\nu,b)}(z)}{\mathbb{W}_{(\mu,a),(\nu,b)}(z)}-1 \right|
=\left|\dfrac{\mathbb{W}'_{(\mu,a),(\nu,b)}(z)-\dfrac{\mathbb{W}_{(\mu,a),(\nu,b)}(z)}{z}}{\dfrac{\mathbb{W}_{(\mu,a),(\nu,b)}(z)}{z}} \right|\\
&< \left[\frac{\Gamma(a)\Gamma(b)}{\Gamma(a+\mu)\Gamma(b+\nu)}+\frac{2(e-1)\Gamma(a)\Gamma(b)}{\Gamma(a+2\mu)\Gamma(b+2\nu)}\right]
\left[ 1- \frac{\Gamma(a)\Gamma(b)}{\Gamma(a+\mu)\Gamma(b+\nu)}  - \frac{(e-1)\Gamma(a)\Gamma(b)}{\Gamma(a+2\mu)\Gamma(b+2\nu)}  \right]^{-1}\\
&<1,
\end{align*}

under the given hypothesis. This proves the theorem.
\end{proof}

\begin{theorem}\label{KT2}
Let $a,b,\mu$ and $\nu$ be positive real numbers. Also suppose that the following conditions hold:
\begin{displaymath}
(H_2):\left\{ \begin{array}{ll}
(i)& \frac{\Gamma(a+2\mu)}{\Gamma(a+3\mu)}<  \frac{3\Gamma(b+3\nu)}{8\Gamma(b+2\nu)},\\
(ii) & \frac{\Gamma(a+\mu)\Gamma(a+3\mu)}{\Gamma^2(a+2\mu)}<\frac{16\Gamma^2(b+2\nu)}{9\Gamma(b+\nu)\Gamma(b+3\nu)},\\
(iii)& \frac{2\Gamma(a)\Gamma(b)}{\Gamma(a+\mu)\Gamma(b+\nu)}-\frac{3(1-e)\Gamma(a)\Gamma(b)}{\Gamma(a+2\mu)\Gamma(b+2\nu)}<1.
\end{array} \right.
\end{displaymath}
Then the function $\mathbb{W}_{(\mu,a),(\nu,b)}(z)$ is convex in $\mathbb{D}_{\frac{1}{2}}.$
\end{theorem}
\begin{proof}Let $z\in\mathbb{D}$, then we get
\begin{equation}\label{K2}
\begin{split}
\left|\mathbb{W}_{(\mu,a),(\nu,b)}^\prime(z)-1\right|&\leq\sum_{k=1}^\infty\frac{(k+1)\Gamma(a)\Gamma(b)}{\Gamma(a+\mu k)\Gamma(b+k\nu)}\\
&=\sum_{k=0}^\infty\frac{(k+2)k!\Gamma(a)\Gamma(b)}{k!\Gamma(a+\mu+\mu k)\Gamma(b+\nu+k\nu)}\\
&=\Gamma(a)\Gamma(b){}_2\Psi_3\Big[ \begin{array}{c}(1, 1), (3,1)\\(2,1),(a+\mu, \mu), (b+\nu, \nu)
			   \end{array} \Big|1\Big].
\end{split}
\end{equation}
In this case,
$$\psi_0=\frac{2}{\Gamma(a+\mu)\Gamma(b+\nu)}, \psi_1=\frac{3}{\Gamma(a+2\mu)\Gamma(b+2\nu)},\;\textrm{and}\;\psi_2=\frac{8}{\Gamma(a+3\mu)\Gamma(b+3\nu)}.$$
Hence, the conditions $``(H_2): (i), (ii)"$ imply that
\begin{equation}\label{K3}
{}_2\Psi_3\Big[ \begin{array}{c}(1, 1), (3,1)\\(2,1),(a+\mu, \mu), (b+\nu, \nu)
			   \end{array} \Big|1\Big]\leq\frac{2}{\Gamma(a+\mu)\Gamma(b+\nu)}-\frac{3(1-e)}{\Gamma(a+2\mu)\Gamma(b+2\nu)}.
				\end{equation}
Hence, combining the hypotheses $``(H_2): (iii)"$, (\ref{K2}) and (\ref{K3}), we have
$$\left|\mathbb{W}_{(\mu,a),(\nu,b)}^\prime(z)-1\right|<1.$$
Therefore, the function $\mathbb{W}_{(\mu,a),(\nu,b)}(z)$ is convex on $\mathbb{D}_{\frac{1}{2}},$ by means of Lemma \ref{lm:convex-D1/2}.
\end{proof}

\begin{corollary}\label{CKT2} Let $a>\frac{2}{7}.$  If $b>\frac{2a+20}{7a-2},$ then the function $\mathbb{W}_{(1,a),(1,b)}(z)$ is convex on $\mathbb{D}_{\frac{1}{2}}.$
\end{corollary}

\begin{proof}
Setting $\mu=\nu=1$ in Theorem \ref{KT2},  we observe that the condition $``(H_2): (i)"$ holds true for all $a,b>0.$ Moreover, it is clear that the condition $``(H_2): (ii)"$ is equivalent to the following inequality
$$b>\frac{2a+20}{7a-2}:=f_1(a),\;\textrm{when}\;a>\frac{2}{7}.$$
Straightforward calculation yields that the assumption $"(H_2): (iii)"$ is equivalent to
$$b>\frac{-(a^2-a-2)+\sqrt{(a^2-a-1)^2+4a(a+1)(2a-1+3e)}}{2a(a+1)}:=g_1(a).$$
By using that fact that the functions $f_1(a)$ and $g_1(a)$ are decreasing on $(2/7,\infty)$ such that
$$\lim_{a\rightarrow (\frac{2}{7})^+} f_1(a)=\infty,\;\lim_{a\rightarrow \infty} f_1(a)=\frac{2}{7}, \lim_{a\rightarrow (\frac{2}{7})^+} g_1(a)\approx 9.63\;\;\textrm{and}\;\;\lim_{a\rightarrow\infty} g_1(a)=0.$$
Therefore
$$b>\max_{a>\frac{2}{7}}(f_1(a), g_1(a))=f_1(a).$$
\end{proof}

Putting $a=14$ in Corollary \ref{CKT2}, we obtain the following example.
\begin{example}
If $b>\frac{1}{2},$ then the function $\mathbb{W}_{(1,14),(1,b)}(z)$ is convex on $\mathbb{D}_{\frac{1}{2}}.$
\end{example}

\begin{theorem} Assume that the conditions of the above Theorem $(H_2): (i), (ii)$ hold true. Also, suppose that
$$\frac{2\Gamma(a)\Gamma(b)}{\Gamma(a+\mu)\Gamma(b+\nu)}-\frac{3(1-e)\Gamma(a)\Gamma(b)}{\Gamma(a+2\mu)\Gamma(b+2\nu)}<\frac{2}{\sqrt{5}}.$$
Then the function $\mathbb{W}_{(\mu,a),(\nu,b)}(z)$ is starlike in $\mathbb{D}.$
\end{theorem}
\begin{proof}The proof is similar to the proof of the above Theorem \ref{KT2} when we used Lemma \ref{lm:starlike-derivative-mocannu}, we omit the details.
\end{proof}
\begin{theorem}\label{KT4}
Suppose that $a,b,\mu$ and $\nu$ are positive real numbers. Assume that the following conditions hold:
\begin{displaymath}
(H_3):\left\{ \begin{array}{ll}
(i)& \frac{\Gamma(a+2\mu)}{\Gamma(a+3\mu)}<  \frac{\Gamma(b+3\nu)}{2\Gamma(b+2\nu)},\\
(ii) & \frac{\Gamma(a+\mu)\Gamma(a+3\mu)}{\Gamma^2(a+2\mu)}<\frac{2\Gamma^2(b+2\nu)}{\Gamma(b+\nu)\Gamma(b+3\nu)},\\
(iii)& \frac{\Gamma(a)\Gamma(b)}{\Gamma(a+\mu)\Gamma(b+\nu)}-\frac{(1-e)\Gamma(a)\Gamma(b)}{\Gamma(a+2\mu)\Gamma(b+2\nu)}<1.
\end{array} \right.
\end{displaymath}
Then the function $\mathbb{W}_{(\mu,a),(\nu,b)}(z)$ is starlike in $\mathbb{D}_{\frac{1}{2}}.$
\end{theorem}
\begin{proof}
It can be verified that
\begin{equation}\label{KKK}
\begin{split}
\left|\frac{\mathbb{W}_{(\mu,a),(\nu,b)}(z)}{z}-1\right|&\leq \sum_{k=0}^\infty\frac{\Gamma(a)\Gamma(b)}{\Gamma(a+\mu+k\mu)\Gamma(b+\nu+k\nu)}\\
&=\Gamma(a)\Gamma(b){}_1\Psi_2\Big[ \begin{array}{c}(1, 1)\\(a+\mu, \mu), (b+\nu, \nu)
			   \end{array} \Big|1\Big].
\end{split}
\end{equation}
In this case,
$$\psi_0=\frac{1}{\Gamma(a+\mu)\Gamma(b+\nu)}, \psi_1=\frac{1}{\Gamma(a+2\mu)\Gamma(b+2\nu)},\;\textrm{and}\;\psi_2=\frac{2}{\Gamma(a+3\mu)\Gamma(b+3\nu)},$$
which is equivalent to $``(H_3): (i), (ii)".$  This implies that
$$
{}_1\Psi_2\Big[ \begin{array}{c}(1, 1)\\(a+\mu, \mu), (b+\nu, \nu)
			   \end{array} \Big|1\Big]\leq\frac{1}{\Gamma(a+\mu)\Gamma(b+\nu)}-\frac{(1-e)}{\Gamma(a+2\mu)\Gamma(b+2\nu)}.$$
Using the above inequality (\ref{KKK}) and hypothesis $``(H_3): (iii)"$, we obtain		
$$\left|\frac{\mathbb{W}_{(\mu,a),(\nu,b)}(z)}{z}-1\right|<1,$$
for all $z\in\mathbb{D},$ which implies that the function $\mathbb{W}_{(\mu,a),(\nu,b)}(z)$ is starlike in $\mathbb{D}_{\frac{1}{2}},$
by Lemma \ref{lm:univalent-D1/2}.
\end{proof}

\begin{corollary} Let $a,b>0.$ If $ab>2$, then $\mathbb{W}_{(1,a),(1,b)}(z)$ is starlike in $\mathbb{D}_{\frac{1}{2}}.$
\end{corollary}

\begin{proof}
Set $\mu=\nu=1$ in Theorem \ref{KT4}. Then the condition $``(H_3): (i)"$ is valid for each $a,b>0$ and the assumption $``(H_3): (ii)"$ holds true for all $ab>2.$ Further, the hypothesis $``(H_3): (iii)"$ is equivalent to the following inequality:
$$b>\frac{1-a^2+\sqrt{(a^2-1)^2+4(a+e)(a^2+a)}}{2a(a+1)}.$$
Consequently,
$$b>\max_{a>0}\left(\frac{2}{a}, \frac{1-a^2+\sqrt{(a^2-1)^2+4(a+e)(a^2+a)}}{2a(a+1)}\right)=\frac{2}{a}.$$
\end{proof}

\begin{example} If $b>2,$ then the function  $\mathbb{W}_{(1,1),(1,b)}(z)$ is starlike in $\mathbb{D}_{\frac{1}{2}}.$
\end{example}
\begin{example}The function $\mathbb{W}_{(1,\sqrt{2}),(1,\sqrt{3})}(z)$ is starlike in $\mathbb{D}_{\frac{1}{2}}.$
\end{example}

\begin{theorem}\label{YYY5}Let $a, b, \nu$ and $\mu$ be positive real numbers and $z\in\mathbb{D}.$ If the following conditions hold
\begin{displaymath}
(H_4):\left\{ \begin{array}{ll}
(i)& \frac{\Gamma(a+\mu)}{\Gamma(a+2\mu)}<\frac{\Gamma(b+2\nu)}{4\Gamma(b+\nu)},
\\
(ii) & \frac{\Gamma(a)\Gamma(a+2\mu)}{\Gamma^2(a+\mu)}<\frac{4\Gamma^2(b+2\nu)}{3\Gamma(b+\nu)\Gamma(b+3\nu)}
,\\
(iii)& \frac{\Gamma(a)\Gamma(b)}{\Gamma(a+\mu)\Gamma(b+\nu)}-\frac{3(1-e)\Gamma(a)\Gamma(b)}{\Gamma(a+2\mu)\Gamma(b+2\nu)}<\frac{1}{4},
\end{array} \right.
\end{displaymath}
then the function $\mathbb{W}_{(\mu,a),(\nu,b)}(z)\in$ UCV.
\end{theorem}
\begin{proof} Let $z\in\mathbb{D}.$ It is easy to see that
\begin{equation*}
\begin{split}
\left|z\mathbb{W}_{(\mu,a),(\nu,b)}^{\prime\prime}(z)\right|&=\Gamma(a)\Gamma(b)\left|{}_1\Psi_2\left[^{\;\;\;(3,1)}_{(a+\mu, \mu),\; (b+\nu, \nu)}\big|z\right]\right|\\
&\leq \Gamma(a)\Gamma(b)\;{}_1\Psi_2\left[^{\;\;\;(3,1)}_{(a+\mu, \mu),\; (b+\nu, \nu)}\big|1\right].
\end{split}
\end{equation*}
Taking into consideration (\ref{fi}) and under the following assumptions
$$(H_{4}^{'}):\frac{\Gamma(a+\mu)}{\Gamma(a+2\mu)}< \frac{\Gamma(b+2\nu)}{4\Gamma(b+\nu)}\;\;{\rm and}\;\;\frac{\Gamma(a+\mu)\Gamma(a+3\mu)}{\Gamma^2(a+2\mu)}<\frac{4\Gamma^2(b+2\nu)}{3\Gamma(b+\nu)\Gamma(b+3\nu)},$$
we get
\begin{equation}\label{1818}
\left|z\mathbb{W}_{(\mu,a),(\nu,b)}^{\prime\prime}(z)\right|\leq \frac{2\Gamma(a)\Gamma(b)}{\Gamma(a+\mu)\Gamma(b+\nu)}-\frac{6(1-e)\Gamma(a)\Gamma(b)}{\Gamma(a+2\mu)\Gamma(b+2\nu)}.
\end{equation}
However, for all $z\in\mathbb{D}$, we have
\begin{equation}\label{mo789}
\begin{split}
\left|\mathbb{W}_{(\mu,a),(\nu,b)}^{\prime}(z)\right|&=\left|1+\sum_{k=1}^\infty\frac{(k+1)\Gamma(a)\Gamma(b) z^k}{\Gamma(a+k\mu)\Gamma(b+k\nu)}\right|\\
&\geq 1-\sum_{k=1}^\infty\frac{(k+1)\Gamma(a)\Gamma(b) z^k}{\Gamma(a+k\mu)\Gamma(b+k\nu)}\\
&=1-\sum_{k=0}^\infty\frac{(k+2)\Gamma(a)\Gamma(b) }{\Gamma(a+\mu+k\mu)\Gamma(b+\nu+k\nu)}\\
&=1-\Gamma(a)\Gamma(b){}_2\Psi_3\left[^{\;\;(1, 1), (3,1)}_{(2, 1), (a+\mu, \mu),\; (b+\nu, \nu)}\big|1\right].
\end{split}
\end{equation}
Using \eqref{fi}, under the following assumptions
$$(H_{4}^{''}): \frac{\Gamma(a+2\mu)}{\Gamma(a+3\mu)}<\frac{3\Gamma(b+3\nu)}{8\Gamma(b+2\nu)}\;\;{\rm and}\;\;\frac{\Gamma(a+\mu)\Gamma(a+3\mu)}{\Gamma^2(a+2\mu)}<\frac{16\Gamma^2(b+2\nu)}{9\Gamma(b+\nu)\Gamma(b+3\nu)},$$
we obtain
\begin{equation*}
{}_2\Psi_3\left[^{\;\;(1, 1), (3,1)}_{(2, 1), (a+\mu, \mu),\; (b+\nu, \nu)}\big|1\right]\leq\frac{2}{\Gamma(a+\mu)\Gamma(b+\nu)}+\frac{3(e-1)}{\Gamma(a+2\mu)\Gamma(b+2\nu)}.
\end{equation*}
In view of the above inequality and (\ref{mo789}), we have
\begin{equation}\label{5555}
\left|\mathbb{W}_{(\mu,a),(\nu,b)}^{\prime}(z)\right|\geq1-\frac{2\Gamma(a)\Gamma(b)}{\Gamma(a+\mu)\Gamma(b+\nu)}+\frac{3(e-1)\Gamma(a)\Gamma(b)}{\Gamma(a+2\mu)\Gamma(b+2\nu)}.
\end{equation}
We notice that the right hand side of the above inequality is positive under the assumption $(H_4): (iii).$ Now, combining (\ref{1818}), (\ref{5555}) and $``(H_4): (iii)"$, we get
$$\left|\frac{\mathbb{W}_{(\mu,a),(\nu,b)}^{\prime\prime}(z)}{\mathbb{W}_{(\mu,a),(\nu,b)}^{\prime}(z)}\right|<1\;\;{\rm for\;\; all}\;\;z\in\mathbb{D}.$$
Hence, the first assertions of Lemma \ref{lm:ucv-Sp} completes the proof of Theorem \ref{YYY5}.

\end{proof}

\begin{corollary}Let $a>3.$ If $b>\frac{2a+6}{a-3}$, then the function $\mathbb{W}_{(1,a),(1,b)}(z)\in UCV.$
\end{corollary}
\begin{proof}
Let $\mu=\nu=1$ in the above Theorem. Then the condition $(H_4): (i)"$ is valid for all $ab+a+b>3.$ The second condition of $(H_4)$ is equivalent to $b(a-3)>6+2a.$ Finally, the condition $``(H_4): (iii)"$ holds true for all
$$b>\frac{-(a^2-3a-4)+\sqrt{(a^2-3a-4)^2+4(4a+12e-8)}}{2a(a+1)}.$$
Consequently,
\begin{equation*}
\begin{split}
b&>\max_{a>3}\left(\frac{2a+6}{a-3}, \frac{-(a^2-3a-4)+\sqrt{(a^2-3a-4)^2+4(4a+12e-8)}}{2a(a+1)}\right)\\&=\frac{2a+6}{a-3}.
\end{split}
\end{equation*}
\end{proof}
\section{Further results on starlikeness, convexity and close-to-convexity}\label{sec-old-method}
In this section, we provide some alternative conditions for starlikeness, convexity and uniform convexity of $\mathbb{W}_{(\mu,a),(\nu,b)}(z),$ which will be useful to discuss close-to-convexity (univalency) of $\mathbb{W}_{(\mu,a),(\nu,b)}(z)$ in $\mathbb{D}.$

\begin{theorem}\label{thm:starlike-D}
Let $a,b,\mu,\nu\geq1$  be such that $b\geq\phi(a)=\dfrac{3a+2}{a(a+1)}.$ Then
\begin{enumerate}[\rm(i)]
\item  $\mathbb{W}_{(\mu,a),(\nu,b)}(z)$ is starlike in $\mathbb{D}$.
\item $\mathbb{W}_{(\mu,a),(\nu,b)}(z)$ is close-to-convex with respect to the starlike function $\mathbb{W}_{(\mu,a),(1,b)}(z)$ in  $\mathbb{D}$.
\end{enumerate}
\end{theorem}
\begin{proof}
\begin{enumerate}[(i)]
\item
Let $p(z)=\dfrac{z\mathbb{W}'_{(\mu,a),(\nu,b)}(z)}{\mathbb{W}_{(\mu,a),(\nu,b)}(z)}$, $z\in\mathbb{D}$. Then $p(z)$ is analytic in $\mathbb{D}$ and
$p(0)=1$. To prove that $\Re(p(z))>0$, it is enough to show that $|p(z)-1|<1$. It is well-known that
$$\Gamma(a+k)\leq \Gamma(a+k\mu),\quad k\in\mathbb{N},\;a,\mu>1.$$
Therefore,
\begin{align}\label{eq:ratio-Gamma-a}
\dfrac{\Gamma(a)}{\Gamma(a+k\mu)}\leq \dfrac{\Gamma(a)}{\Gamma(a+k)}
=\dfrac{1}{a(a+1)\cdots(a+k-1)}.
\end{align}
Similarly, we have
\begin{align}\label{eq:ratio-Gamma-b}
\dfrac{\Gamma(b)}{\Gamma(b+k\nu)}\leq \dfrac{\Gamma(b)}{\Gamma(b+k)}
=\dfrac{1}{b(b+1)\cdots(b+k-1)}.
\end{align}

With the help of Lemma \ref{lm:two-ineq} and the above inequalities \eqref{eq:ratio-Gamma-a} and \eqref{eq:ratio-Gamma-b}, we obtain
\begin{equation*}
\begin{split}
\left|\mathbb{W}^\prime_{(\mu,a),(\nu,b)}(z)-\dfrac{\mathbb{W}_{(\mu,a),(\nu,b)}(z)}{z}\right|
&=\left|\sum_{k=1}^{\infty}\dfrac{k\Gamma(a)\Gamma(b)z^k}{\Gamma(a+k\mu)\Gamma(b+k\nu)} \right|\\
&\leq \sum_{k=1}^{\infty}\dfrac{k}{ab(a+1)(b+1)\cdots(a+k-1)(b+k-1)}\\
&<\dfrac{1}{ab}+\dfrac{1}{ab}\sum_{k=2}^{\infty}\dfrac{1}{(a+1)^{k-2}(b+1)^{k-1}}\\
&=\dfrac{1}{ab}+\dfrac{1}{ab(b+1)}\sum_{k=0}^{\infty}\left\{\dfrac{1}{(a+1)(b+1)}\right\}^{k}\\
&=\dfrac{(a+1)(b+1)+a}{ab\{(a+1)(b+1)-1\}}.
\end{split}
\end{equation*}
Again, we have
\begin{equation*}
\begin{split}
\left| \dfrac{\mathbb{W}_{(\mu,a),(\nu,b)}(z)}{z}\right|
&\geq 1-\left| \sum_{k=1}^{\infty}\dfrac{\Gamma(a)\Gamma(b)z^k}{\Gamma(a+k\mu)(b+k\nu)} \right|\\
&\geq1-\sum_{k=1}^{\infty}\dfrac{1}{ab(a+1)(b+1)\cdots(a+k-1)(b+k-1)}\\
&>1-\dfrac{1}{ab}\sum_{k=0}^{\infty}\left\{\dfrac{1}{(a+1)(b+1)}\right\}^k
=1-\dfrac{(a+1)(b+1)}{ab(ab+a+b)}\\
&=\dfrac{ab(a+b+ab)-(a+1)(b+1)}{ab(a+b+ab)}.
\end{split}
\end{equation*}
Therefore,
\begin{align*}
\left|p(z)-1 \right|&=\left|\dfrac{z\mathbb{W}'_{(\mu,a),(\nu,b)}(z)}{\mathbb{W}_{(\mu,a),(\nu,b)}(z)}-1 \right|
=\left|\dfrac{\mathbb{W}'_{(\mu,a),(\nu,b)}(z)-\dfrac{\mathbb{W}_{(\mu,a),(\nu,b)}(z)}{z}}{\dfrac{\mathbb{W}_{(\mu,a),(\nu,b)}(z)}{z}} \right|\\
&<\dfrac{(a+1)(b+1)+a}{ab(a+b+ab)-(a+1)(b+1)}.
\end{align*}
Under the given condition, $\frac{(a+1)(b+1)+a}{ab(a+b+ab)-(a+1)(b+1)}\leq1$. This shows that
$\mathbb{W}_{(\mu,a),(\nu,b)}(z)$ is starlike in $\mathbb{D}$ and consequently the part (i) of this theorem is proved.

\item
Now, we proceed to prove the part (ii). For this,
we have to show that there exists a function $h\in\mathcal{S}^{*}$ such that
$$
\Re\left(\dfrac{z\mathbb{W}'_{(\mu,a),(\nu,b)}(z)}{h(z)} \right)>0,\quad z\in \mathbb{D},
$$
which can be proved by showing that
$$
\left|\dfrac{z\mathbb{W}'_{(\mu,a),(\nu,b)}(z)}{h(z)}-1\right|<1,\quad z\in\mathbb{D}.
$$
Using part (i) of this Theorem \ref{thm:starlike-D}, we can observe that $\mathbb{W}_{(\mu,a),(1,b)}(z)$ is starlike in $\mathbb{D}$
under the given hypothesis. Again using (\ref{Eq35}) and (\ref{Eq36}), we obtain
\begin{align*}
\left|\mathbb{W}'_{(\mu,a),(\nu,b)}(z)-\dfrac{\mathbb{W}_{(\mu,a),(1,b)}(z)}{z}\right|
&<\sum_{k=1}^{\infty}\left|\dfrac{(k+1)\Gamma(a)\Gamma(b)}{\Gamma(a+k\mu)\Gamma(b+k\nu)}-\dfrac{\Gamma(a)\Gamma(b)}{\Gamma(a+k)\Gamma(b+k)} \right|\\
&\leq \sum_{k=1}^{\infty}\left|\dfrac{k\Gamma(a)\Gamma(b)}{\Gamma(a+k)\Gamma(b+k)}\right|\\
&\leq \sum_{k=1}^{\infty}\dfrac{k}{ab(a+1)(b+1)\cdots(a+k-1)(b+k-1)}\\
&<\dfrac{1}{ab}+\dfrac{1}{ab}\sum_{k=2}^{\infty}\dfrac{1}{(a+1)^{k-2}(b+1)^{k-1}}\\
&=\dfrac{(a+1)(b+1)+a}{ab\{(a+1)(b+1)-1\}}
\end{align*}
and
\begin{align*}
\left|\dfrac{\mathbb{W}_{(\mu,a),(1,b)}(z)}{z}\right|&\geq 1-\left|\dfrac{\Gamma(a)\Gamma(b)z^k}{\Gamma(a+k)\Gamma(b+k)}\right|\\
&\geq1-\sum_{k=1}^{\infty}\dfrac{1}{ab(a+1)(b+1)\cdots(a+k-1)(b+k-1)}\\
&>1-\dfrac{1}{ab}\sum_{k=0}^{\infty}\left\{\dfrac{1}{(a+1)(b+1)}\right\}^k
=1-\dfrac{(a+1)(b+1)}{ab(ab+a+b)}\\
&=\dfrac{ab(a+b+ab)-(a+1)(b+1)}{ab(a+b+ab)}.
\end{align*}
Using the above inequalities and given conditions, we have
\begin{align*}
\left|\dfrac{z\mathbb{W}'_{(\mu,a),(\nu,b)}(z)}{\mathbb{W}_{(\mu,a),(1,b)}(z)}-1 \right|
&=\left|\dfrac{\mathbb{W}'_{(\mu,a),(\nu,b)}(z)-\dfrac{\mathbb{W}_{(\mu,a),(1,b)}(z)}{z}}{\dfrac{\mathbb{W}_{(\mu,a),(1,b)}(z)}{z}} \right|\\
&<\dfrac{(a+1)(b+1)+a}{ab(a+b+ab)-(a+1)(b+1)}\leq 1,
\end{align*}
which implies that $\Re\left(\frac{z\mathbb{W}'_{(\mu,a),(\mu,b)}(z)}{\mathbb{W}_{(\mu,a),(1,b)}(z)} \right)>0$.
Consequently, $\mathbb{W}_{(\mu,a),(\mu,b)}(z)$ is close-to-convex with respect to the starlike function $\mathbb{W}_{(\mu,a),(1,b)}(z)$
in $\mathbb{D},$ under the given hypothesis.
\end{enumerate}
\end{proof}

\begin{theorem}
Let $a,b,\mu,\nu\geq1$ and $0\leq \eta<1$ be such that $b\geq\psi(a,\eta)$, where
\begin{align*}
\psi(a,\eta)&=\dfrac{(a+1)+(1-\eta)(a+1-a^2)}{2a(1-\eta)(a+1)}\\
&+\dfrac{\sqrt{\{(a+1)+(1-\eta)(a+1-a^2)\}^2-4a(1-\eta)(a+1)\{(1-\eta)(a+1)+2a+1\}}}{2a(1-\eta)(a+1)}.
\end{align*}
Then $\mathbb{W}_{(\mu,a),(\nu,b)}(z)\in\mathcal{S}^{*}(\eta)$ for each $z\in\mathbb{D}.$
\end{theorem}
\begin{proof}
From the proof of Theorem \ref{thm:starlike-D}, it can be observed that $\mathbb{W}_{(\mu,a),(\nu,b)}(z)$ is starlike function
of order $\eta$, if $\frac{(a+1)(b+1)+a}{ab(a+b+ab)-(a+1)(b+1)}\leq1-\eta$, which is a consequence of the given condition.
\end{proof}

Using Lemma \ref{lm:ucv-Sp} and the similar technique as in Theorem \ref{thm:starlike-D}, the following theorem can be obtained.

\begin{theorem}\label{thm:S-p}
Let $a,b,\mu,\nu\geq1$  be such that $b\geq\tau(a)$, where
$$
\tau(a)=\dfrac{3(a+1)-a^2+\sqrt{a^4+14a^3+35a^2+30a+9}}{2a(a+1)}.
$$
Then  $\mathbb{W}_{(\mu,a),(\nu,b)}(z)\in\mathcal{S}_p$ for each $z\in\mathbb{D}.$
\end{theorem}

\begin{theorem}\label{TH4}
Let $a,b,\mu,\nu\geq1$. If $b\geq\phi_1(a)$, where
$$
\phi_1(a)=\dfrac{(a+1-a^2)+\sqrt{a^4+2a^3+7a^2+6a+1}}{2a(a+1)},
$$
then  $\mathbb{W}_{(\mu,a),(\nu,b)}(z)$ is starlike in $\mathbb{D}_{1/2}$.
\end{theorem}
\begin{proof}
With the help of Lemma \ref{lm:two-ineq} and the inequalities \eqref{eq:ratio-Gamma-a} and \eqref{eq:ratio-Gamma-b}, we obtain
\begin{align*}
\left|\dfrac{\mathbb{W}_{(\mu,a),(\nu,b)}(z)}{z} -1\right|
&<\sum_{k=1}^{\infty}\dfrac{1}{ab(a+1)(b+1)\cdots(a+k-1)(b+k-1)}\\
&\leq\dfrac{1}{ab}\sum_{k=0}^{\infty}\left\{\dfrac{1}{(a+1)(b+1)}\right\}^k
=\dfrac{(a+1)(b+1)}{ab(ab+a+b)}.
\end{align*}
Using the given condition and Lemma \ref{lm:univalent-D1/2}, we conclude that $\mathbb{W}_{(\mu,a),(\nu,b)}(z)$ is starlike in $\mathbb{D}_{1/2}$, under the given hypothesis.
\end{proof}
\begin{theorem}\label{thm:convex-D1/2}
Let $a,b,\mu,\nu\geq1$.
\begin{enumerate}[\normalfont(i)]
\item If $b\geq\psi_1(a)$, where
$$
\psi_1(a)=\dfrac{(3-a)+\sqrt{a^2+2a+9}}{2a}.
$$
Then  $\mathbb{W}_{(\mu,a),(\nu,b)}(z)$ is convex in $\mathbb{D}_{1/2}$.
\item If $b\geq\psi_2(a)$, where
$$
\psi_2(a)=\dfrac{\sqrt{5}(2a+3)-2a^2+\sqrt{4a^4+8\sqrt{5}a^3+20(1+\sqrt{5})a^2+4(15+4\sqrt{5})a+45}}{4a(a+1)}.
$$
Then  $\mathbb{W}_{(\mu,a),(\nu,b)}(z)$ is starlike in $\mathbb{D}$.
\end{enumerate}

\end{theorem}
\begin{proof} Using inequalities \eqref{eq:ratio-Gamma-a} and \ref{eq:ratio-Gamma-b} and Lemma \ref{lm:two-ineq}, we have
\begin{align*}
\left|\mathbb{W}'_{(\mu,a),(\nu,b)}(z)-1\right|&<\sum_{k=1}^{\infty}\dfrac{(k+1)\Gamma(a)\Gamma(b)}{\Gamma(a+k\mu)\Gamma(b+k\nu)}\\
&\leq \sum_{k=1}^{\infty}\dfrac{k+1}{ab(a+1)(b+1)\cdots(a+k-1)(b+k-1)}\\
&=\dfrac{1}{ab}+\sum_{k=2}^{\infty}\dfrac{k}{ab(a+1)(b+1)\cdots(a+k-1)(b+k-1)}\\
&\quad+\sum_{k=1}^{\infty}\dfrac{1}{ab(a+1)(b+1)\cdots(a+k-1)(b+k-1)}\\
&\leq \dfrac{1}{ab}+\sum_{k=2}^{\infty}\dfrac{1}{ab(a+1)^{k-1}(b+1)^{k-2}}+\sum_{k=1}^{\infty}\dfrac{1}{ab(a+1)^{k-1}(b+1)^{k-1}}\\
&=\dfrac{1}{ab}+\dfrac{(a+2)}{ab(a+1)}\sum_{k=0}^{\infty}\left\{\dfrac{1}{(a+1)(b+1)}\right\}^k\\
&=\dfrac{2a(b+1)+3b+2}{ab\{(a+1)(b+1)-1\}}.
\end{align*}
Using Lemma \ref{lm:convex-D1/2} and given condition (i), we can conclude that $\mathbb{W}_{(\mu,a),(\nu,b)}(z)$ is convex in $\mathbb{D}_{1/2}$.
Further using the given condition (ii) and with the help of Lemma \ref{lm:starlike-derivative-mocannu}, it is easy to show that
$\mathbb{W}_{(\mu,a),(\nu,b)}(z)$ is starlike in $\mathbb{D}$. Hence the theorem is proved.
\end{proof}

If we set $a=\mu=\nu=1$ in the normalized four parameters Wright function \eqref{eq:normalized-Wright-four-parameter},
then we obtain the normalized form of a class of functions involving the confluent hypergeometric function as follows:
\begin{align*}
\mathbb{F}(b,z)=z\sum_{k=0}^\infty \dfrac{\Gamma(b)}{\Gamma(b+k)}\dfrac{z^{k}}{k!}=z\times{}_0F_1(-;b;z),
\end{align*}
where ${}_0F_1(-;b;z)$ is the confluent hypergeometric function \cite{abramowitz64,Andrews-Askey-Roy-Book-specialFunctions-1999}.
Following corollary provides a set of sufficient conditions for $\mathbb{F}(b,z)$ to be  starlike and convex in the open unit disk $\mathbb{D}$
and in $\mathbb{D}_{1/2}$.

\begin{corollary}The following assertions hold true:
\begin{enumerate}[\normalfont(i)]
\item If $b\geq\frac{5}{4}$, then $\mathbb{F}(b,z)$ is starlike in the open unit disk $\mathbb{D}$.\\
\item If $b\geq \frac{(5+\sqrt{89})}{4}$, then $\mathbb{F}(b,z)\in \mathcal{S}_p$.
\item If $b\geq\frac{(1+\sqrt{17})}{4}$, then  $\mathbb{F}(b,z)$ is starlike in $\mathbb{D}_{1/2}$.\\
\item If $b\geq1+\sqrt{3}$, then  $\mathbb{F}(b,z)$ is convex on $\mathbb{D}_{1/2}$.
\end{enumerate}
\end{corollary}
\begin{proof}
Part (i) can be proved using  Theorem \ref{thm:starlike-D}. Part (ii) is a consequence of Theorem \ref{thm:S-p}.
Using part (i) of Theorem \ref{TH4}, part (iii) of this corollary can be obtained. Finally using Theorem \ref{thm:convex-D1/2},
part (iv) can be proved.
\end{proof}

\begin{figure}[http!]
\centering
\begin{subfigure}[b]{0.4\textwidth}
                \includegraphics[width=1\columnwidth]{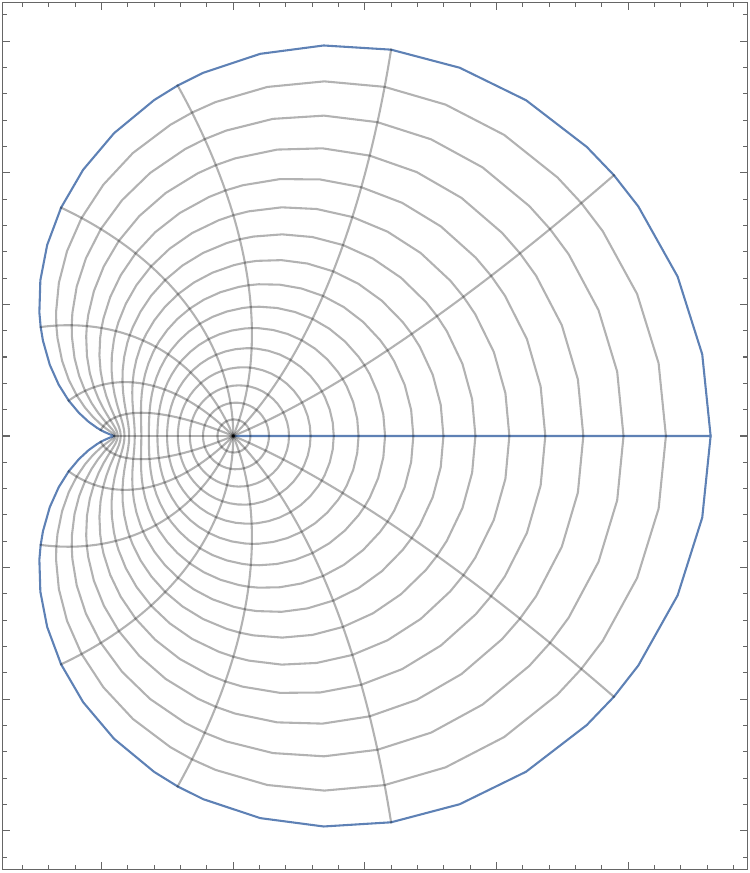}
                \caption{$\mathbb{F}(3/2,z)$ for $z\in\mathbb{D}$.}
        \end{subfigure}\hfill
\begin{subfigure}[b]{0.4\textwidth}
                \includegraphics[width=1\columnwidth]{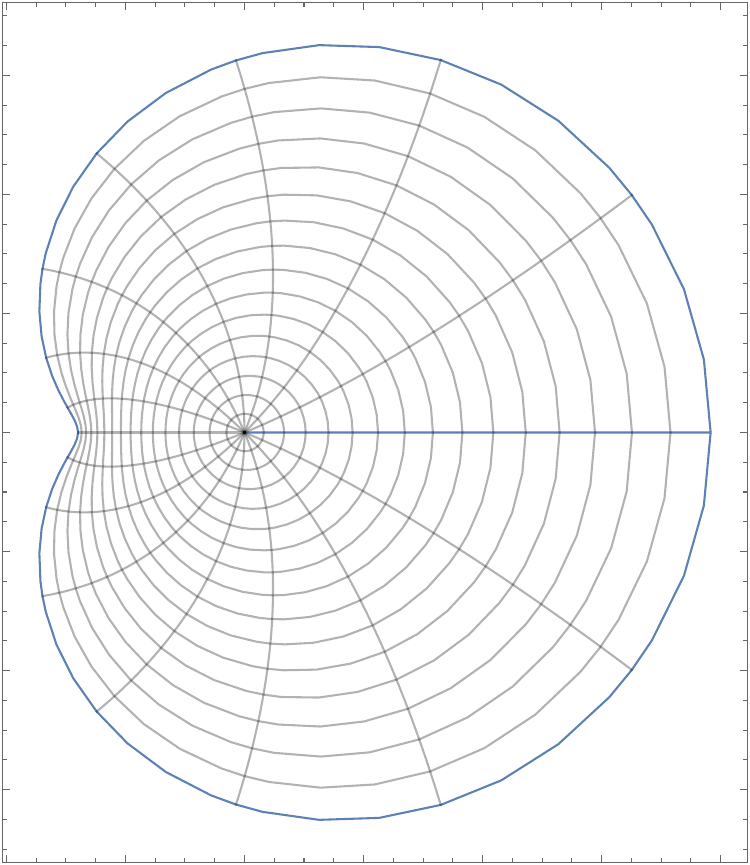}
                \caption{ $\mathbb{F}(1,z)$ for $z\in\mathbb{D}_{1/2}$.}
        \end{subfigure}\hfill \newline\newline
        \begin{subfigure}[b]{0.4\textwidth}
                \includegraphics[width=1\columnwidth]{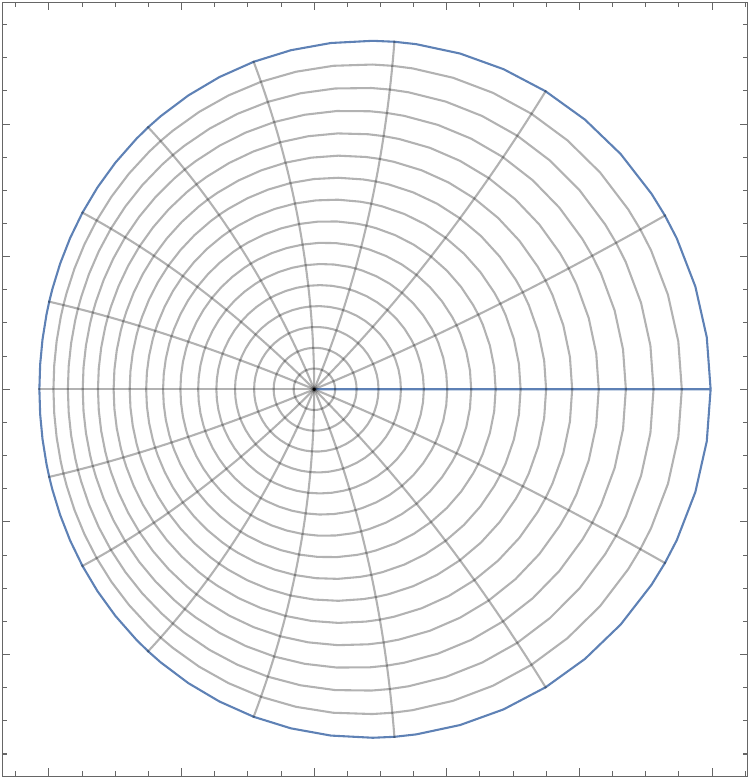}
                \caption{$\mathbb{F}(1+\sqrt{3},z)$ for $z\in\mathbb{D}_{1/2}$.}
        \end{subfigure}\hfill
         \begin{subfigure}[b]{0.4\textwidth}
                \includegraphics[width=1\columnwidth]{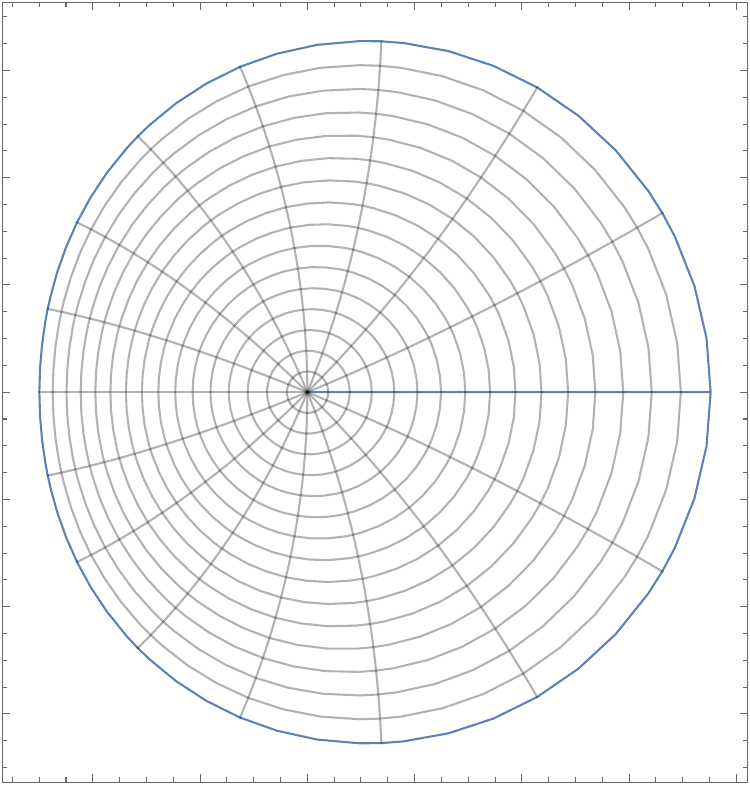}
                \caption{$\mathbb{W}_{(1,\sqrt{2}),(1,\sqrt{3})}(z)$ for $z\in\mathbb{D}_{1/2}$.}
        \end{subfigure}%
\caption{Mapping of $\mathbb{F}(b,z)$ and $\mathbb{W}_{(\mu,a),(\nu,b)}(z)$ over $\mathbb{D}$ and $\mathbb{D}_{1/2}$.}
\end{figure}

\section{Results related to $\mathcal{W}_{(\mu,a),(\nu,b)}(z)$}

\begin{theorem}\label{thm:ineq-Wright-hypergeom}
For any $a,b,\mu,\nu\geq1$ and $z\in\mathbb{D},$ the following inequality holds:
$$
\left|\mathbb{W}_{(\mu,a),(\nu,b)}(z) \right|\leq|z|\times {}_1F_2(1;a,b;|z|),
$$
where $\displaystyle{}_1F_2(c;a,b;|z|)=\sum_{k=0}^{\infty}\dfrac{(c)_k}{(a)_k(b)_k}\dfrac{|z|^{k}}{k!}$ is hypergeometric function with
$(\alpha)_k$ as pochhammer symbol defined as $\displaystyle(\alpha)_0=1,(\alpha)_k=\alpha(\alpha+1)\cdots(\alpha+k-1)$.
\end{theorem}
\begin{proof}By using \eqref{Eq35} and \eqref{Eq36}, we obtain
\begin{align*}
\left|\mathbb{W}_{(\mu,a),(\nu,b)}(z) \right|
&=\left|\sum_{k=0}^{\infty}\dfrac{\Gamma(a)\Gamma(b)z^{k+1}}{\Gamma(k\mu+a)\Gamma(k\nu+b)} \right|\\
&\leq|z|\sum_{k=0}^{\infty}\left|\dfrac{\Gamma(a)\Gamma(b)z^{k}}{\Gamma(k\mu+a)\Gamma(k\nu+b)} \right|\\
&\leq|z|\sum_{k=0}^{\infty}\dfrac{|z|^{k}}{ab(a+1)(b+1)\cdots (a+k-1)(b+k-1)} \\
&=|z|\sum_{k=0}^{\infty}\dfrac{(1)_k}{(a)_k(b)_k}\dfrac{|z|^{k}}{k!}=|z|\times {}_1F_2(1;a,b;|z|).
\end{align*}
\end{proof}
Using \eqref{1.2} and Theorem \ref{thm:ineq-Wright-hypergeom}, following corollary can be obtained.
\begin{corollary}
For any $a,b,\mu,\nu\geq1$ and $z\in\mathbb{D}$ the following inequality holds:
$$
\left|\mathcal{W}_{(\mu,a),(\nu,b)}(z) \right| \leq \dfrac{|z|}{\Gamma(a)\Gamma(b)}\times {}_1F_2(1;a,b;|z|),
$$
where $\displaystyle{}_1F_2(1;a,b;|z|)$ is a hypergeometric function.
\end{corollary}

\section{Conclusion}
Our research discusses about some geometric properties (such as starlikeness, convexity, uniform convexity and close-to-convexity) of four parameters Wright function. In addition, geometric properties of the partial sums of this function are studied.
As applications,  we obtain certain geometric properties of normalized Bessel function of the first kind and two parameters Wright function
as given below.
\begin{enumerate}
\item It can be noted that for $a=\mu=\nu=1$, $b=\beta+1$ and $z=-z(z\in\mathbb{D})$, we have a normalization of the Bessel function of first kind
\cite{Baricz-Book-Bessel-Function1994} $\mathcal{J}_{\beta}(z)$ defined as \cite{Prajapat-Wright-Function-ITSF2015}:
$$
\mathbb{J}_{\beta}(z)=\mathbb{W}_{(1,\beta+1),(1,1)}(-z)=\Gamma(\beta+1)z^{1-\beta/2}\mathcal{J}_{\beta}(2\sqrt{z}).
$$
\begin{enumerate}[\normalfont(i)]
\item Using Theorem \ref{thm:starlike-D}, we claim that $\mathbb{J}_{\beta}(z)$ is starlike in $\mathbb{D}$ if $\beta\geq 3/2$, which is a sharper than the lower bound $(\sqrt{3})$ available in \cite{Prajapat-Wright-Function-ITSF2015}.

\item Using Theorem \ref{thm:convex-D1/2}, we obtain $\mathbb{J}_{\beta}(z)$
is convex in $\mathbb{D}_{1/2}$ if $\beta\geq \sqrt{3}$, which is the same condition available in \cite{Prajapat-Wright-Function-ITSF2015}.

\item Theorem \ref{thm:S-p} helps us to conclude that $\mathbb{J}_{\beta}(z)\in \mathcal{S}_p$ if $\beta\geq \frac{(1+\sqrt{89})}{4}$.
\end{enumerate}
\item For $a=\mu=1$, the four parameters Wright function reduces to two parameters Wright function
$$
W_{b,\nu}(z)=\mathcal{W}_{(1,1),(\nu,b)}(z)=\sum_{k=0}^{\infty}\dfrac{z^k}{k!\Gamma(b+k\nu)},
$$
and the corresponding normalized two parameters Wright function can be defined as
$$
\mathbb{W}_{b,\nu}(z)=\mathbb{W}_{(1,1),(\nu,b)}(z)=\Gamma(b)W_{b,\nu}(z)=\sum_{k=0}^{\infty}\dfrac{\Gamma(b)z^k}{k!\Gamma(b+k\nu)}.
$$
\begin{enumerate}[\normalfont(i)]
\item Using Theorem \ref{thm:starlike-D}, we claim that if $\nu\geq1$ and $b\geq 5/2$,
then $\mathbb{W}_{b,\nu}(z)$ is starlike in $\mathbb{D}$. This lower bound of $b$ is sharper than the lower bound $(b\geq 1+\sqrt{3})$ available
in \cite{Prajapat-Wright-Function-ITSF2015}.

\item From Theorem \ref{thm:convex-D1/2}, we have if $\nu\geq1$ and $b\geq(1+\sqrt{3})$, then $\mathbb{W}_{b,\nu}(z)$ is convex in
$\mathbb{D}_{1/2}$, which is the same condition available in \cite{Prajapat-Wright-Function-ITSF2015}.

\item From Theorem \ref{thm:starlike-D}, we obtain, if $\nu\geq1$ and $b\geq 5/2$, then $\mathbb{W}_{b,\nu}(z)$ is
close-to-convex with respect to $\mathbb{W}_{1,\nu}(z)$ in $\mathbb{D}$.

\item Using Theorem \ref{thm:S-p}, we conclude that $\mathbb{W}_{b,\nu}(z)\in \mathcal{S}_p$ if $b\geq \frac{(5+\sqrt{89})}{4}$.

\item Setting $a=\mu=1$ and $b=4$ in Theorem \ref{KT2}, we obtain that $\mathbb{W}_{4,\nu}(z)$ is convex in $\mathbb{D}_{\frac{1}{2}}$ if $\nu\in [0.76,0.95]$. Putting $a=\mu=1$ and $b=2$ in Theorem \ref{KT4}, we have $\mathbb{W}_{2,\nu}(z)$ is starlike in $\mathbb{D}_{\frac{1}{2}}$ if
    $\nu\in[0.645, 0.999]$. Similarly, we can verify that the other results obtained in Section \ref{sec-new-method} will be useful to discuss the geometric properties of $\mathbb{W}_{b,\nu}(z)$ when $0<\nu<1$. In literature, various results involving geometric properties of $\mathbb{W}_{b,\nu}(z)$ are available \cite{Prajapat-Wright-Function-ITSF2015} with the condition that $b,\nu\geq1$. But the results obtained in Section \ref{sec-new-method} discuss the case $0<\nu<1$. On the other hand, the results obtained in Section \ref{sec-old-method}, will be helpful to discuss the geometric properties of $\mathbb{W}_{b,\nu}(z)$ when $b,\nu\geq1$.
\end{enumerate}
\end{enumerate}

\end{document}